\documentclass[letterpaper]{article}
\usepackage[margin=3cm]{geometry}
\usepackage{layout}
\usepackage{color}
\usepackage{tabu}
\usepackage{upgreek}
\usepackage{amsmath}
\usepackage{mathrsfs}
 \usepackage{setspace}
 \usepackage{amssymb}
 \usepackage{graphicx}
 \usepackage{tikz-cd}
\usepackage{bm}
\usepackage{tikz}
 \usepackage{enumitem}
\usepackage{mathrsfs}
 \usepackage{amsthm}
\usepackage{hyperref}                   

\usepackage{cleveref} 
\usepackage{nomencl}
\newtheorem{thm}{Theorem}[section]
\newtheorem{co}[thm]{Corollary}
\newtheorem{lem}[thm]{Lemma}
\newtheorem{pr}[thm]{Proposition}

\theoremstyle{definition}
\newtheorem{re}[thm]{Remark}

\newtheorem{ex}[thm]{Example}
\newtheorem{prob}[thm]{Problem}

\crefname{section}{§}{§§}
\Crefname{section}{§}{§§}

\newcommand{\ci}{\mathbf{1}}

\newcommand{\bs}{\backslash}
\newcommand{\ld}{\lambda}
\newcommand{\Ld}{\Lambda}
\newcommand{\del}{\delta}

\newcommand{\ve}{\varepsilon}
\newcommand{\si}{\sigma}

\newcommand{\al}{\alpha}

\newcommand{\cd}{\cdot}

\newcommand{\be}{\beta}

\newcommand{\la}[2]{\left\langle #1, #2 \right\rangle}
\newcommand{\od}{\odot}
\newcommand{\cch}{\overline{\text{conv}}}
\newcommand{\supp}{\text{supp}}
\newcommand{\cc}{C_c}

\newcommand{\mc}{\mathcal{M}(G)}

\newcommand{\lc}{\mathcal{L}_c(G)}
\newcommand{\lgg}{\mathcal{L}(G)}

\newcommand{\msc}{\mathcal{M}_{\si}(G)}
\newcommand{\mw}{\mathcal{M}_{w}(G)}

\begin{document}
		
		{ \title{\large Abstract harmonic analysis on locally compact right topological groups}

			\author{\normalsize Prachi Loliencar\\
				\normalsize Department of Mathematics and Statistics\\
				\normalsize University of Alberta\\
				\normalsize				Edmonton, Alberta\\
				\normalsize				Canada T6G 2G1\\
				\texttt{\normalsize lolienca@ualberta.ca}
				\\
			}
			\date{}
			\maketitle
		}
		
		\begin{abstract}
		Analytic properties of right topological groups have been extensively studied in the compact admissible case (i.e when the group has a dense topological center). This was inspired by the existence of a Haar measure on such groups. In this paper, we broaden the scope of this work. We give (similar) sufficient conditions for the existence of a Haar measure on locally compact right topological groups and generalize analytic theory to this setting. We then define new measure algebra analogues in the compact setting and use these to completely characterize the existence of a Haar measure, producing sufficient conditions that do not rely on admissibility. 
			\end{abstract}
		
		
		\bigskip

		A right topological group is a group equipped with a topology, $(G,\tau)$, satisfying continuity of right multiplication, i.e the maps $G\to G$, $g\mapsto gx$ are continuous for all $x\in G$. The topological center of such a group, denoted by $\Ld(G)$, refers then to those elements $x\in G$ for which left multiplication i.e. the map $g\mapsto xg$ is also continuous.  A right topological group is said to be admissible when its topological center is dense. Interest in these groups initially arose in topological dynamics where Ellis gave a beautiful theorem proving that compact Hausdorff admissible right topological (CHART) groups arise naturally from distal flows \cite{ellis}. However, this discovery has also drawn an interest in these groups from the perspective of abstract harmonic analysis.
		
		 Abstract harmonic analysis on locally compact topological groups heavily relies on the existence of a Haar measure. A Borel measure $\mu$ on a right topological group $G$ is said to be right (resp. left) invariant if $\mu(Eg)= \mu(E)$ ($\mu(gE)=\mu(E)$) for all $g\in G$ (resp. $g\in \Ld(G)$) and for all Borel subsets $E\subset G$. Such a measure is said to be a (right) Haar measure if it is a Radon measure. When $G$ is compact, it is  additionally assumed to be a probability measure. In \cite{milnespym}, Pym and Milnes proved the existence of a unique Haar measure on CHART groups, or more generally, on groups having a strong normal system of subgroups (for definition see \Cref{strongnorm}). Following this, Lau and Loy conducted analysis on these groups \cite{lauloy1}\cite{lauloy2} and defined measure and Fourier algebra analogues. All of this work builds on the fundamental algebro-topological theory on these groups developed by Namioka in \cite{namioka}.
		
		In this paper, our first goal is to generalize this existing theory to locally compact right topological groups.  In particular, we often deal with $\si$-locally compact right topological groups for added structure. Secondly, we define new measure algebra analogues and characterize the existence of a Haar measure in terms of their properties in the compact setting. 
		
		 In \Cref{preliminaries}, we review the existing theory on right topological groups and introduce some notation. Then, in \Cref{haarsec}, we prove the existence of a Haar measure on locally compact right topological groups possessing a compact strong normal system of subgroups. Further, in \Cref{func}, we introduce the $\si\si$-topology and generalize some of the results in \cite{lauloy1}, describing properties of various function algebras, including almost periodic functions and the Fourier and Fourier-Stieltjes algebras. \Cref{measuresec} deals with measure algebra analogues on right topological groups. Here we discuss how some of these measure algebras are fundamental in characterizing the existence of a Haar measure on compact right topological groups. In the process, we provide new sufficient conditions for the existence of a Haar measure on a compact group that need not be admissible. Lastly, we conclude in \Cref{openprob} with some examples and open problems left to be worked on.
		
		\section{Preliminaries}\label{preliminaries}
		
		A group with topology $(G,\tau)$ is said to be {\it semitopological} if its multiplication map $m: G\times G\to G$, $(x,y)\mapsto x^{-1}y$ is separately continuous. It is said to be {\it topological} if the multiplication $m$ is jointly continuous and additionally, the inverse map $G\to G$, $x\mapsto x^{-1}$ is also continuous. Throughout, we denote by $R_g$, $L_g$ the right and left translation maps on $G$ for each $g\in G$. For any function $f$ on $G$, we also use $R_gf$ to denote the right translate of $f$ by $g\in G$.
		
		 Ellis proved the following famous theorem \cite{ellis1};
		
		\begin{thm}[Ellis] Every locally compact Hausdorff semitopological group is topological. \end{thm}
		
		As a result, any interest in generalizing theory of locally compact topological groups naturally shifts to groups with one-sided continuity of multiplication.  Namioka later gave a more general result on separate and joint continuity \cite{namioka2},
		
		\begin{thm}[Namioka] \label{separatecont} Let $X$ be locally compact Hausdorff space and $(G,\tau)$ be a topological group acting on $X$. If the following hold:
			\begin{itemize}
				\item $G$ is locally compact
				\item $G\times X\to X$, $(g,x)\to gx$ is separately continuous
				\item $G$ is right topological		
			\end{itemize}
			Then, $G\times X \to X$ is jointly continuous.
		\end{thm}
		
		For any topological space $X$, we shall denote by $C_b(X)$, the continuous bounded complex-valued functions on $X$, equipped with the usual supremum norm. Further, $C_c(X)$ and $C_0(X)$ will refer to the continuous functions with compact support, and the continuous functions vanishing at infinity, respectively on $X$. $M(X)$ will then denote the dual of $C_c(X)$, the complex Radon measures on $X$.\\
		
		Namioka's pioneering work on right topological groups primarily relies on his introduction of the $\si$-topology. Given a right topological group $(G,\tau)$, the $\si$-topology on G is defined to be quotient topology given by the map
		\begin{align*}
		\phi: (G,\tau)\times(G,\tau) &\to G\\
		(x,y)&\mapsto x^{-1}y
		\end{align*}
		The symbol $\si$ is intended by Namioka to indicate ``symmetry" and is justified by the following theorem \cite{namioka}:
		\begin{thm}[Namioka] \label{si}
			Let $(G,\tau)$ be a right topological group and $\si$ be its $\si$-topology. Then,
			\begin{enumerate}
				\item $\si\subset \tau$ and equality holds if and only if $(G,\tau)$ is topological
				\item $(G,\si)$ is a semitopological group with a continuous inverse map
				\item $(G,\si)$ is $T_1$ if and only if $(G,\tau)$ is Hausdorff. Further, $(G,\si)$ is Hausdorff if and only if $(G,\tau)$ is topological
			\end{enumerate}
		\end{thm}
		Observe that if $G$ is locally compact, a lack of being Hausdorff is all that holds $(G,\si)$ (and therefore $(G,\tau)$) from being topological. This also produces interesting examples of semitopological groups.\\
		
		One of the advantages of working with admissible right topological groups is the following theorem \cite{namioka}.
		
		\begin{thm}[Namioka] \label{open}
			Let $G$ be a admissible right topological group and let $\mathcal{U}$ be a base of open neighbourhoods of $e$ in $\tau$. Then, the following holds:
			\begin{enumerate}
				\item The quotient map $\phi: (G,\tau)\times (G,\tau) \to (G,\si)$ is open
				\item The family $\{U^{-1}U\mid U\in \mathcal{U}\}$ forms a base of $\si$-open neighbourhoods of $e$ in $(G,\si)$
			\end{enumerate}	
		\end{thm}
		
		For general right topological groups, it becomes hard to explicitly find the open sets in the $\si$-topology. Unlike the compact case, where many results in \cite{namioka} and \cite{lauloy1} hold for general, non-admissible groups, in the locally compact case, we heavily rely on the above result to produce analogues.

		Following Namioka, for subgroups $K$ of $G$, we denote by $G/K$ the left cosets of $K$ i.e. $\{xK\mid x\in G\}$. Namioka showed the following; 
		\begin{pr}[Namioka] \label{hausdorff}
			$G/K$ is Hausdorff if and only if $K$ is $\si$-closed.
		\end{pr}
		
		Recall that in the topological case it suffices for $K$ to be closed. We will use $(G/K, \tau)$ and $(G/K,\si)$ to mean the quotient topology on $G/K$ induced by $(G,\tau)$ and $(G,\si)$ respectively.\\
		
		Let $L\subset G$ be a closed normal subgroup. Let us denote by $(L,\si)$, the topology induced by $(G,\si)$. We warn the reader that this does not always coincide with the (finer) $\si$-topology of $(L,\tau)$ (work on this topology may be found in \cite{moran1}). We define $N(L)$ to be the intersection of all $\si$-closed $\si$-neighbourhoods of $e$ in $G$. The following is a fundamental result of Namioka \cite{namioka} which we generalize to locally compact topological groups in \Cref{haarsec}
		
		\begin{pr} \label{ng} Let $G$ be a compact Hausdorff right topological group. Then,
			\begin{enumerate}
			\item $N(L)$ is a $\si$-closed normal subgroup of $L$ 
			\item $(L/N(L),\tau)=(L/N(L),\si)$ and the resulting group is a compact topological group
			\item The action map 
			\begin{align*}
			(G/N(L),\tau) \times (L/N(L),\tau) &\to (G/N(L),\tau)\\
			([x],[y]) &\mapsto [xy]
			\end{align*}
			is jointly continuous.
			\end{enumerate}
		\end{pr}

		Via a quotient, one therefore obtains a topological group from the compact right topological group $G$. Namioka proved the non-triviality of this group in the following result for the countably admissible case ($\Ld(G)$ has a countable subsemigroup that is dense in $G$) \cite{namioka}, while Pym and Milnes generalized it to arbitrary compact admissible groups by Pym and Milnes \cite{milnespym2}.
		
		\begin{thm} \label{moorsnam}
			Suppose $G$ is a CHART group and $L\triangleleft G$ satisfies $L\neq \{e\}$. Then, $N(L)\neq L$.
		\end{thm}
		
	Central to the existence of a Haar measure on compact right topological groups is the idea of a strong normal system of subgroups. A right topological group $(G,\tau)$ is said to have such a system if there exists a family $\{L_\xi\}_{\xi<\xi_0}$ of $\si$-closed normal subgroups of $G$, indexed by some ordinal $\xi_0>0$, satisfying the following conditions: 
	\begin{enumerate} \label{strongnorm}
		\item $L_{\xi_0}= G$, $L_0= \{e\}$, $L_\xi \supset L_{\xi+1}$ and for a limit ordinal $\xi<\xi_0$, $L_\xi=\cap_{\eta<\xi} L_\eta$;
		\item $L_{\xi}/L_{\xi+1}$ is a compact Hausdorff topological group
		\item The map 
		\begin{align*}
		G/L_{\xi+1} \times L_{\xi}/L_{\xi+1}&\to G/L_{\xi+1}\\
		([x],[y])&\mapsto [xy]
		\end{align*}
		is jointly continuous
		\end{enumerate}
		\bigskip
		
		Pym and Milnes exploited \Cref{ng} and \Cref{moorsnam} to obtain the following nice theorems \cite{milnespym}\cite{milnespym2}.
		
		\begin{thm} \label{oldhaar}
			Every compact Hausdorff right topological group with a strong normal system of subgroups has a unique right-invariant Haar measure that is additionally left-invariant with respect to the topological center.
		\end{thm}
		
		\begin{thm} \label{strongnormal}
			Every CHART group has a strong normal system of subgroups.
		\end{thm}
		
		As a result, one obtains as a corollary the existence of a unique right-invariant Haar measure on all CHART groups. In this paper, we generalize \Cref{oldhaar}. We are currently unable to verify if \Cref{strongnormal} holds for $\si$-locally compact groups. 
		
		\section{Existence of Haar measure} \label{haarsec}
		
		We primarily consider $\si$-locally compact groups due to the following convenient result
		
		\begin{lem}[Open mapping theorem for right topological groups] \label{openmap} Suppose $(G,\tau)$ a $\si$-compact right topological group, and $H$ is a Hausdorff right topological group with the Baire property. If $\phi: G\to H$ is a continuous surjective homomorphism, then it is open.
		\end{lem}
		
		The proof of this result follows similarly to the classical theorem on locally compact topological groups (see proof of 3 of \Cref{loc5} for a similar technique).\\
		
		Despite the $\si$-topology being non-Hausdorff, we obtain the following result
		
		\begin{lem} \label{baire}
			If $(G,\tau)$ is locally compact hausdorff admissible, then $(G,\si)$ is a locally compact semitopological group that has the Baire property.
		\end{lem}
		
		\begin{proof}
			As $(G,\tau )$ is admissible, by \Cref{open}, the continuous quotient map $\phi: (G\times G,\tau\times \tau) \to (G,\si)$ is open. It follows that $(G,\si)$ is locally compact.
Suppose $\{U_n\}_{n\in\mathbb{N}}\subset \si$ are $\si$-open sets dense in the $\si$ topology. We claim that $\cap_{n\in\mathbb{N}} U_n$ is dense in $(G,\si)$. Indeed, $\{\phi^{-1}(U_n)\}_{n\in\mathbb{N}}$ are open and dense in $G\times G$, so that by $G\times G$ being locally compact Hausdorff (whence Baire), $\cap_{n\in\mathbb{N}}\phi^{-1}(U_n)$ is dense in $G\times G$. Applying $\phi$, by surjectivity, $\cap_{n\in\mathbb{N}}{U_n}$ is dense in $(G,\si)$ as well.
		\end{proof}
		
		\bigskip

		We may now prove the locally compact version of \Cref{ng}.
		
		\begin{pr} \label{loc5}
			Let $(G,\tau)$ be a compact or $\si$-locally compact admissible Hausdorff right topological group. If $L\subset G$ is a $\si$-closed subgroup, then 
			\begin{enumerate}
				\item $N(L)$ is a normal subgroup  that is closed in $(L,\si)$
				\item $(L/N(L), \tau) = (L/N(L),\si)$ is a $\si$-compact Hausdorff topological group
				\item $G/N(L) \times L/N(L) \to G/N(L)$, $([x],[y]) \mapsto [xy]$ is jointly continuous.
			\end{enumerate}
		\end{pr}

		\begin{proof}
			The compact case is shown in \cite{namioka}. We assume $G$ is $\si$-locally compact admissible. 
			
			Part 1 follows easily as $N(L)$ is the intersection of $\si$-closed neighbourhoods, and normality is guaranteed by Corollary 1.1 in \cite{namioka}. 
			
			By \Cref{hausdorff}, $(L/N(L),\si)$ is Hausdorff, while local compactness of the space follows from \Cref{baire} as $(L,\si)\subset (G,\si)$ is $\si$-closed. By \Cref{openmap} then, the continuous identity map $(L/N(L), \tau)\to (L/N(L),\si)$ is a homeomorphism and  $(L/N(L), \tau)=(L/N(L),\si)$ as claimed. As $(L/N(L),\si)$ is a locally compact Hausdorff semitopological group, by Ellis' theorem, 2 holds.
			
			Lastly, we prove 3. By continuity of right multiplication, it is clear that $G/N(L) \times L/N(L) \to G/N(L)$, $([x],[y]) \mapsto [xy]$ is continuous in the first variable. We will thus show continuity in the second variable. Let us fix $x\in G$. Then, the map $l_x: L/N(L) \to xL/N(L)$, $[y]\to [xy]$ is a $\si$-$\si$ homeomorphism. 
			We claim that the following composition of open maps is continuous;
			
			
			\begin{center}
				\begin{tikzpicture}
				\node (A) at (-2,1) {$(L/N(L),\tau)$};
				\node (B) at (-2,0) {$(L/N(L),\si)$};
				\node (C) at (1.5,0) {$(xL/N(L),\si)$};
				\node (D) at (5,0) {$(xL/N(L),\tau)$};
				\node (b) at (-2,-0.5) {$[y]$};
				\node (c) at (1.5,-0.5) {$[xy]$};
				\node (d) at (5,-0.5) {$[xy]$};
				\draw[->] (B) -- (C);
				\draw[->] (C) -- (D);
				\draw[|->] (-0.75,-0.5) -- (0.15,-0.5);
				\draw[|->] (2.85,-0.5) -- (3.65,-0.5);
				\draw[-] (-2.05,0.7) -- (-2.05,0.3);
				\draw[-] (-1.95,0.7) -- (-1.95,0.3);
				\end{tikzpicture}
			\end{center}

			%
			We claim that this composition is continuous. Indeed consider the continuous surjective inverse map $p:(xL/N(L),\tau) \to  (L/N(L),\tau)$, $[xy] \mapsto [y]$. Observe that $xL/N(L)$ being $\si$-closed in $G/N(L)$, is also $\tau$-closed, whence $\si$-compact in $\tau$.  By $\si$-compactness there exists a sequence of compact sets $\{C_n\}_{n\in\mathbb{N}}$ of $xL/N(L)$, such that $\cup_{n\in\mathbb{N}} C_n = xL/N(L)$. Then, $\{p(C_n)\}_{n\in\mathbb{N}}$ are clearly compact sets whose union is L/N(L). By the Baire property of $(L/N(L), \tau)$, there exists a compact set $C\in \{C_n\}_{n\in\mathbb{N}}$, such that $p(C)$ has an interior point $p(c)$, with $c\in C$.  If  $\{g_{\al}\}\subset L/N(L)$ is a net such that $p(g_{\al}) \to p(g)$ for some $g\in L/N(L)$, it follows by continuity of right multiplication of $(G/N(L),\tau)$ that $p(g_{\al})p(g)^{-1}p(c) = p(g_{\al}g^{-1}c) \to p(c)$ and $p(C)$ contains a tail of $\{p(g_{\al}g^{-1}c)\}$. However,  $p|a_{C}:(C,\tau) \to (x^{-1}C,\tau)$ is a homeomorphism (compact into Hausdorff space) so that applying $[p|_C]^{-1}$, $g_{\al}g^{-1}c\to c$. By continuity of right multiplication, clearly $g_{\al}\to g$ and thus, $p^{-1}$ is continuous. This concludes the proof of the claim. We have thus shown that the map in 3. is separately continuous. Joint continuity of the map then follows from Namioka's \Cref{separatecont}.
		\end{proof}
		
		\bigskip
		
		In order to prove the existence of a Haar measure on locally compact right topological groups, we use a technique similar to that in Pym and Milnes' original proof. We begin by giving the following lemma and provide a proof for completeness.
		
		\begin{lem} \label{loc1}
			Suppose $L$, $M$ are normal subgroups of $G$ satisfying the following conditions
			\begin{itemize}
				\item $L$, $M$ are compact in $(G,\si)$
				\item $L/M$ is a topological group
				\item $G/M\times L/M \to G/M$, $([x],[y])\to [xy]$ is jointly continuous.
			\end{itemize}
			Let $\nu$ be the unique Haar measure on $L/M$. Then, the map $\phi:\cc(G/M)\to \cc(G/L), \text{ } f \mapsto \int_{L/M} \ f(\cd t) \ d\ld(t)$, is a positive retraction such that if $f\in \cc(G/M)$ and $\text{\emph{supp}}(f)=K/M$, then $\text{\emph{supp}}(\phi(f))\subset K/L$. Moreover, for each $g\in G$, $R_g\circ\phi =  \phi\circ R_g$.
		\end{lem}

		\begin{proof}
			By the continuity of the map $G/M\times L/M \to G/M$, $([x],[y])\to [xy]$, it is clear that the map $\int_{L/M} f(\cd t) \ d\ld(t)$ is well-defined. By the left-invariance of $\nu$, $\phi(f)\in C(G/L)$. Moreover, if $f\in \cc(G/L) \subset \cc(G/M)$, 	
			$$	s\mapsto \phi(f)(s) = \int_{L/M}f(st) d\ld(t) =  \int_{L/M}f(s) d\ld(t)= f(s)$$	so that the map is a retraction.
			
			If $\supp(f) = K/M$, then, for $s\in G$,
			$$|\phi(f)(s)| = \left|\int_{L/M}f(st) \  d\ld(t)\right| \leq {\|f\|}_{\infty} \int_{L/M} \ci_{K/M}(st) \ d\ld(t)=\ld((s^{-1}K\cap L)/M) $$
			
			Here $s^{-1}K\cap L$ is non-empty implies that 
			$s\in KL$. In other words, the above quantity is non-zero, only when $s\in KL$, so that if we consider $\phi(f)\in C(G/L)$, it follows that $[s]\in K/L$. This proves our claim and that $f\in \cc(G/L)$.
			
			Now, note that $L/M \to G/M \to L/M$, $[t] \mapsto [gt]\mapsto [gtg^{-1}]$ is continuous, the first map being continuous by assumption, and the second being continuous due to $G/L$ being a right topological group. It is easy to observe that the map $C(L/M)\to \mathbb{C}$, $f\mapsto \int f(gtg^{-1})d\mu(t)$ is right-translation invariant. However, since the Haar measure on $L/M$ is unique, it follows that $\int f(gtg^{-1})d\mu(t) = \int f(t) d\mu(t)$. Using this fact, if $g,s\in G$,
			
			\begin{align*}
			R_g\circ\phi(f)(s) =\phi(f)(sg)= \int_{L/M} f(sgt) d\mu(t)
	&= \int_{L/M} f(sgtg^{-1}g)d\mu(t)\\ &=\int_{L/M} f(stg) d\mu(t)\\ &=\int_{L/M}R_gf(st)d\mu(t) \\ &= \phi(R_gf)(s) 
			\end{align*}
			proving the last claim. 
			
		\end{proof}
			
		\bigskip
		
		We also obtain the following generalization from \cite{namioka};
		\begin{pr}
			If $G$ is a $\si$-locally compact admissible hausdorff right topological group and $f: G\to H$ is a homomorphism into a Hausdorff topological group $H$, then $f$ factors through $G/N(G)$.
		\end{pr}
		
		\begin{proof}
			Observe that $f\circ \phi: (G\times G,\tau\times \tau) \to H, (x,y)\mapsto f(x^{-1}y)= f(x)^{-1}f(y)$ is continuous as $f$ is continuous and $H$ is topological. By definition this implies $\si$-continuity of $f$. By Corollary 1.1 in \cite{namioka}, the factorization follows.
		\end{proof}

		The main theorem of this section follows.
		\begin{thm} \label{haarmeas}
			Suppose $G$ is a locally compact Hausdorff right topological group that has a compact strong normal system of subgroups. Then, $G$ has a right invariant Haar measure.
		\end{thm}
		
		\begin{proof}
			Let $\{L_\xi\}_{\xi\leq \xi_0}$ be the given strong normal system of subgroups. For each $\xi > 0$, we denote by 
			$\phi_\xi: \cc(G/L_{\xi+1})\to \cc(G/L_{\xi})$, the map from \Cref{loc1}, and by $\nu_{\xi}$ the Haar measure on $L_{\xi}/L_{\xi+1}$. Using transfinite induction, we construct for each $\xi>0$, a linear functional $\psi_\xi: \cc(G/L_\xi)\to \mathbb{C}$, satisfying the following conditions:
			\begin{enumerate}
				\item $\psi_\xi$ is positive
				\item $\psi_\xi$ is right-invariant
				\item $\psi_{\xi}(f) = \psi_{\eta}(f)$, for all $f\in \cc(G/L_\eta)$, for all $0<\eta\leq \xi$.
				\item For each $K$ compact, $\psi_\xi(K/L_\xi)\leq\psi_1(K/L_1)<\infty$
			\end{enumerate}
			By the Riesz representation theorem, each $\psi_\xi$ on $\cc(G/L_\xi)$, corresponds to a unique regular Borel  measure on $G/L_\xi$, and 4 implies the existence of a common upper bound for these for every fixed compact set of $G$. We shall show that $\psi_{1}$ is non-zero, so that by 3, it follows that for some $f\in C_c(G/L_1)$,  $0<\psi_{\eta}(f) = \psi_{\xi}(f)$, i.e. $\psi_{\xi}$ is non-zero.

			For the base case, we observe that $G/L_1 = L_0/L_1$ by assumption is a locally compact Hausdorff topological group. Thus, we may fix a Haar measure $\psi_{1}$ on $G/L_1$, so that the map $\psi_{1}: \cc(G/L_1)\to \mathbb{C}$ is the desired linear functional satisfying 1-4. 
			
			Suppose for $\xi<\xi_0$, there exists a functional $\psi_{\xi}: \cc(G/L_\xi) \to \mathbb{C}$ of the desired form. Then, we define $\psi_{\xi+1}: \cc(G/L_{\xi+1})\to \mathbb{C}$ to be given by $\psi_{\xi+1}= \psi_{\xi}\circ \phi_{\xi}$. Positivity and right invariance are clear from \Cref{loc1} and the right invariance of $\psi_{\xi}$.  For any $0<\eta\leq \xi+1$,  $f\in \cc(G/L_\eta)$,  
			$$\psi_{\xi+1}(f) = \psi_{\xi}\circ \phi_\xi (f)= \psi_{\xi}(f)= \psi_\eta(f)$$
			Here, since $\xi+1$ is the smallest ordinal following $\xi$, any ordinal $\eta <\xi+1$, satisfies $\eta \leq \xi$, so that the second equality follows from the retraction property of $\phi_{\xi}$ (\Cref{loc1}), and the third equality follows from the induction assumption. Lastly, for any $f \in \cc(G/L_{\xi+1})$ with support $K$ in $G$, by \Cref{loc1}, $\phi_\xi(f)$ has support $K/L_\xi$, so that $$\psi_{\xi+1}(f)= \psi_{\xi}\circ\phi_{\xi}(f)\leq \|f\|_{\infty}\psi_{\xi}(K/L_\xi)<\psi_{1}(K)$$ and 4 holds for the successor case.	
			
			Suppose $\xi\leq \xi_0$ is a limit ordinal so that $L_\xi = \cap_{\eta<\xi}L_\eta$. Then, consider the subalgebra $D=\cup_{\eta<\xi} \cc(G/L_\eta) \subset \cc(G/L_\xi)\subset C_0(G/L_\xi)$. If $[x]\neq [y]$, for $[x],[y] \in G/L_\xi$, then $x^{-1}y\not\in L_\xi= \cap_{\eta<\xi}L_{\eta}$, so that for some $\eta<\xi$, $[x]\neq [y]$ in $G/L_\eta$. By local compactness then, there exists some $f\in\cc(G/L_\eta)$, such that $f([x])\neq f([y])$. It follows that $D$ separates points in $G/L_{\xi}$. Moreover, it is clear that $D$ is vanishing nowhere. By the Stone-Weierstrass theorem, $D$ is dense in $C_0(G/L_\xi)$.
			
			Now for each compact set $K$, we fix an open neighbourhood $U_K$ of $K$ such that $\overline{U_K}$ is compact. Then, by Urysohn's lemma, for each $\xi\leq \xi_0$, there exists a function $p^K_\xi: G/L_{\xi_n}\to [0,1]$, such that $\ci_{K/L_{\xi_n}}\leq p_\xi^K\leq \ci_{U_K/L_{\xi_n}}$.	Consider any $f\in \cc(G/L_\xi)$ and  let $\{f_n\}_{n\in\mathbb{N}} \subset D$ be such that $f_n\to f$.  Without loss of generality, we assume  ${\|f\|}_{\infty}\leq 1$ so that $\{f_n\}_{n\in\mathbb{N}}$ may be chosen to satisfy ${\|f_n\|}_{\infty}\leq 1$. Let us denote by $\xi_n$, the ordinal corresponding to each $f_n\in \cc(G/L_{\xi_n})$. Suppose $C= \supp(f)$ in $G$.
 Then, $\{f_n p_{\xi_n}^C\}_{n\in\mathbb{N}}\subset D$, which we will write as $\{f_np_{n}^C\}_{n\in\mathbb{N}}$ have supports contained in  $U_C/L_{\xi_n}$, for each $n\in \mathbb{N}$,  and clearly, converge to $f$.
 	Since $\xi=\sup_{\eta<\xi} \eta$, and $\cc(G/L_{\eta_1}) \subset \cc(G/L_{\eta_2})$, for $\eta_1\geq\eta_2$, we may assume that $\{\xi_n\}$ is monotone increasing, so that for $m\geq n$, by the induction assumption,
			\begin{align*} 
			\|\psi_{\xi_n}(p_nf_n)-\psi_{\xi_m}(p_mf_m)\| = \|\psi_{\xi_m}(p_nf_n-p_m f_m)\| &\leq \psi_{\xi_m}(U_C/L_{\xi_m})\|f_n-f_m\|\\&\leq \psi_{\xi_1}(U_C/L_1)\|f_n-f_m\| \to 0 
			\end{align*}
			as $m,n\to \infty$. Thus, the sequence $\{\psi_{\xi_n}(p_nf_n)\}_{n\in\mathbb{N}}$ is Cauchy, so that we define	  $\psi_{\xi}(f)=\lim_{n\in\mathbb{N}}\psi_{\xi_n}(p_nf_n)$. It is easily checked that $\psi$ is well-defined.
			
			That $\psi_{\xi}$ is positive and linear follows from its definition and the induction assumption. Right-invariance of $\psi_{\xi}$ also follows as $\{f_n\}_{n\in\mathbb{N}}$ from above satisfies, $\{ R_gf_n\}_{n\in\mathbb{N}}\subset D$, $R_gf_n \to R_gf$, so that $$\psi_{\xi}(R_gf)=\lim_{n\in\mathbb{N}}\psi_{L_{\xi_n}}(R_gf_n)=  \lim_{n\in\mathbb{N}}\psi_{L_{\xi_n}}(f_n)= \psi_{\xi}(f)$$
			
			Regarding 3, one notes that if $0<\eta\leq \xi$, then, for $f\in \cc(G/L_\eta)$, one may consider an arbitrary sequence $\{\xi_n\}\geq \eta$, and take $f_n=f$ trivially, so that by the induction assumption,
			$$\psi_{\xi}(f)= \lim_{n\in\mathbb{N}}\psi_{\xi_n}(f_n)= \lim_{n\in\mathbb{N}}\psi_{\eta}(f_n)= \psi_{\eta}(f)$$
			
			Lastly, $\psi_\xi(K)\leq\psi_1(K)<\infty$, follows similarly from the induction assumption.
			
			Since $L_{\xi_0} =\{e\}$, there exists a linear functional $\psi_{\xi_0}$ on $C_c(G)$ satisfying the criteria 1-4. This provides the desired Haar measure.
		\end{proof}

		\bigskip 
		
		\begin{thm}
			The Haar measure in theorem \Cref{haarmeas} is unique up to scalar multiplication.
		\end{thm}
		
		\begin{proof}
			 Let $\{\psi_\xi\}_{\xi\leq \xi_0}$ be the Haar measures on $G/L_{\xi}$, $0<\xi\leq \xi_0$ and suppose $\mu$ is a Haar measure on $G$. As before we denote by $\mu_\xi$ the unique Haar measure on $L_{\xi}/L_{\xi+1}$. We shall show that there exists $c>0$ such that $c\mu(f) = \psi_\xi(f)$, for all $f\in \cc(G/L_{\xi})$, for all $\xi\leq\xi_0$ using transfinite induction.
			
			Since $G/L_1$ is a locally compact topological group, and $\mu$ forms a Haar measure on it, it follows that there exists some $c>0$, such that $\mu(f) = c\psi_1(f)$. 
			
			Assume the same induction hypothesis holds for $\xi<
			\xi_0$, with the same constant $c$. Then, for any $f\in \cc(G/L_{\xi+1})$,
			\begin{align*}
			\psi_{\xi+1}(f)= \psi_{\xi}\circ\phi_\xi(f) &= \int \phi_\xi(s) d\mu(s)\\
			&=c\int \int_{L_\xi/L_{\xi+1}} f(st) d\mu_{\xi}(t) d\mu(s)\\  
			&=c\int \int_{L_\xi/L_{\xi+1}} f(s) d\mu_{\xi}(t) d\mu(s)\\  
			&=c\int f(s) d\mu(s)\\ &= c\mu(f)
			\end{align*}
			where we made use of the right-translation invariance of $\mu$. The successor case is hence justified.
			
			Now suppose $\xi$ is a limit ordinal and $f\in \cc(G)$. Recall from the proof of theorem, that $\psi_\xi(f) = \lim_{n\in\mathbb{N}}\psi_{\xi_n}(f_n)$, for $f_n\in \cc(G/L_{\xi_n})$, where  $\cup_{n\in\mathbb{N}}\supp(f_n)
			\subset U$  for some compact set $U\subset G$ and $\|f_n\|\leq \|f\|_{\infty}$ for all $n\in\mathbb{N}$. By our induction hypothesis, therefore, $\psi_\xi(f) =c \lim_{n\in\mathbb{N}}\mu(f_n) = c\mu(f)$, by the dominated convergence theorem. 
			This concludes the proof.
		\end{proof}

			\section{Function algebras} \label{func}

		In this section, we shall discuss classical function algebras in the locally compact right topological group context. In the process we generalize several results of Lau and Loy \cite{lauloy1}\cite{lauloy2}.\\
%
%
%
%
		
		Given a CHART group, one might question when its quotient groups are topological (and thus automatically have a Haar measure). To answer this, we need the $\si\si$-topology of $(G,\si)$.\\
		
		We define the  {$\si\si$ topology} on $G$ to be the $\si$ topology of the compact semi-topological group $(G,\si)$ i.e. the topology induced by the quotient map $(G,\si)\times (G,\si) \to G$, $(g,h)\mapsto g^{-1}h$. By \Cref{si}, $\si\si \subsetneq \si$ and $(G,\si\si)$ is a compact semitopological group with continuous inverse. Furthermore, by \cite{namioka}, Corollary 1.1, $N(G)$ is precisely $\overline{\{e\}}^{\si\si}$.\\
		
		Let $H\subset G$ be a closed normal subgroup. Recall that $(G/H,\tau)$, $(G/H,\si)$ denotes $G/H$ with the quotient topologies induced by $(G,\tau)$ and $(G,\si)$ respectively, where we denote the respective quotient maps by $\pi_\tau$, $\pi_\si$. Let us further define $(G/H, \si_{G/H})$ to be the $\si$ topology induced by $(G/H,\tau)$ i.e. by the quotient map $\phi_{G/H}: (G/H,\tau) \times (G/H,\tau) \to G/H$, $([x],[y])\mapsto [x^{-1}y]$. 
		
		\begin{lem} \label{quoti} Let G be a locally compact right topological group and let $H$ be a $\si$-closed normal subgroup of $G$. Then, $(G/H,\si_{G/H}) = (G/H,\si)$, and $(G/H,\tau) = (G/H,\si)$ holds if and only if $(G/H,\tau)$ is topological.
		\end{lem}
		
		\begin{proof}

%
		To prove the first claim, we consider the following commutative diagram	
				\begin{center}
				\begin{tikzpicture}
				\node (A) at (-2,1) {$(G\times G,\tau\times \tau)$};
				\node (C) at (1.5,1) {$(G,\si)$};
				\node (B) at (-2,-1) {$(G/H\times G/H,\tau\times \tau)$};
				\node (D) at (1.5,-1) {$(G/H,\sigma)$};
				\draw[->] (A) --node[above] {$\phi$}  (C);
				\draw[->] (C) -- node[right] {$\phi_{G/H}$} (D);
				\draw[->] (A) --node[left] {$\pi_{\tau}\times \pi_{\tau}$}  (B);
				\draw[->] (B) -- node[below]{$\pi_{\si}$} (D);
				\draw[->] (A) -- node[below]{$g$} (D);	
				\end{tikzpicture}
			\end{center}
		
		One easily checks that the diagram commutes and a function $f:(G/H,\tau) \to X$, where X is any topological space, is continuous on either of $(G/H, \si_{G/H})$, $(G/H,\si)$ if and only if $f\circ g$ is continuous. It follows that the two topologies coincide.\\
		
		Now by \Cref{si}, $(G/H,\tau)$ is topological if and only if $(G/H,\tau) = (G/H,\si_\tau)$, where the latter space coincides with $(G/H,\si)$. The conclusion follows.
		\end{proof}

		\bigskip
		
		\begin{lem} \label{sisi} Let $G$ be a $\si$-locally compact Hausdorff right topological group and $H\subset G$ be a closed normal subgroup. If $G$ is either compact or admissible, then, $(G/H,\tau)$ is a Hausdorff topological group if and only if $H$ is $\si\si$-closed.
		\end{lem}

		\begin{proof} By \Cref{hausdorff} $(G/H,\tau)$ is Hausdorff if and only if $H$ is $\si$-closed, so that we may restrict the proof to this case. Suppose $H$ is also $\si\si$-closed. Then, by \Cref{baire}, $(G/H,\si)$ is locally compact Hausdorff. However, by \Cref{quoti}, $(G/H,\si)=(G/H,\si_\tau)$, so that $(G/H,\tau) \to (G/H,\si)$, $g\mapsto g$, is a continuous homomorphism from a $\si$-compact  group into a Baire group, thus a homeomorphism by the open mapping theorem. It follows by \Cref{quoti} that $(G/H,\tau)$ is topological. Conversely, if $(G/H,\tau)$ is a Hausdorff topological group, we have $(G/H,\tau) = (G/H,\si)=(G/H,\si_{G/H})$ is Hausdorff so that by \Cref{hausdorff} $H$ must be $\si\si$-closed. This concludes the proof.
		\end{proof}
		
		\bigskip

		In this section we discuss the usual function algebras that have interesting properties in the classical case of locally compact topological groups. The reader is referred to \cite{hewitt}, \cite{kaniuth} for a general theory of these.\\

		Let $A\subset C_0(G)$ be a non-trivial translation-invariant C*-algebra of $C_0(G)$. We then define $$\text{Fix}(A)= \{g\in G\mid L_gf = f \text{, for all }f\in A\}$$
		
		The following result generalizes straightforwardly from Lau with mild modifications.

		\begin{lem}\label{loc2}
			Let $G$ be a locally compact Hausdorff right topological group. If $A\subset C_0(G)$ is a non-trivial translation invariant C*-algebra, then  $F=\text{Fix}(A)$ satisfies the following;
			\begin{enumerate}
				\item $F$ is a closed normal subgroup of $G$
				\item $\tilde{\pi}_F: C_0(G/F) \to A$ is an isometric isomorphism, i.e. $$A=\{f\in C_0(G)\mid L_yf=f\text{ for all }y\in F\}$$
				\item $G/F$ is a locally compact Hausdorff topological group, so that F is $\si$-closed.
			\end{enumerate}
		\end{lem}
		
		We observe here that if G is admissible, a combination of 3 and \Cref{sisi} provides that $F$ is $\si\si$ closed, whence contains $N(G)$.

		\bigskip
		
		As a result of this, we obtain the following
		
		\begin{co} \label{loc3}
			Given a $\si$-locally compact Hausdorff right topological group, if $G$ is either compact or admissible, $$\text{\emph{Fix}}(C_0(G,\si))= N(G)$$ 
		\end{co}
		
		\begin{proof}
			By \Cref{loc2}, for $F= \text{Fix}(C_0(G,\si))$, $(G/F,\tau)$ is a Hausdorff topological group, so that by \Cref{sisi}, $F$ is $\si\si$-closed. However, $N(G) = \overline{\{e\}}^{\sigma\sigma}$ implies $N(G) \subset F$.  
			
			
			
			Suppose that $G$ is compact or admissible. If $N(G)=G$, $F\subset N(G)$ is trivial. 
			If $N(G)\neq G$ on the other hand, for any $g\not\in N(G)$, there exists some $\tilde{f}\in C_0(G/N(G))$ such that $\tilde{f}([g])\neq \tilde{f}([e])$. By \Cref{loc5}, $f=\tilde{f}\circ \pi_{N(G)}\in C(G,\si)$ and separates $e$ and $g$, which implies $g\not\in F$ and thus $N(G)= F$.
			
		\end{proof}

		Let $G$ be a right topological semigroup. We define the left continuous functions on $G$ to be given by 		$$LC(G) = \{f\in C_b(G)\mid L_gf\in C_b(G),\text{ for all }g\in G\}$$
		and the left continuous functions vanishing at infinity to be 		$$LC_0(G)=\{f\in C_0(G)\mid L_gf\in C_0(G),\text{ for all }g\in G\}$$
		Note that the latter definition is somewhat artificial: $LC_0(G)\subset LC(G)\cap C_0(G)$, however, equality is not known, since our left multiplication is not continuous (and thus compact sets need not be preserved by it).\\
		
		For any function $f$ on $G$, we denote by $RO(f)$ the right orbit of $f$, i.e. $\{R_gf\mid g\in G\}$. We may define the standard spaces  $$AP(G)= \{f\in C_b(G)\mid RO(f) \text{ is relatively compact in }C_b(G) \},$$ the almost periodic functions on $G$, and $$WAP(G )= \{f\in C_b(G)\mid RO(f)\text{ is relatively weakly compact in }C_b(G)\},$$ the weakly almost periodic functions on $G$. 
		
		The following proposition is a standard result in harmonic analysis (see Theorem 1.8 of \cite{berglund}).
		
		\begin{lem} \label{loc21}
			Given a non-empty set $S$, and a conjugate closed subspace $E$ of $l_\infty(S)$, $E^*$ is the weak*-closed linear span of $\epsilon(S)$, where $\epsilon: S\to E^*$ is the function sending $s\in S$ to the evaluation functional $\epsilon(s): f\mapsto f(s)$.
		\end{lem}

		\begin{lem}\label{loc211}
			If $G$ is a right topological semigroup, the following equalities hold
			\begin{itemize}
				\item $AP(G) = AP(G_d)\cap C_b(G)$	
				\item $WAP(G)= WAP(G_d)\cap C_b(G)$
			\end{itemize}
		\end{lem}
		
	\begin{proof}
			The first result is obvious. Suppose $f\in l_\infty(G)$. Then, By Grothendieck's double limit theorem (see Theorem A.5 in \cite{berglund}) and \Cref{loc21}, for any sequences $\{g_n\}$,$\{h_n\}\subset G$, the following middle limits are equal when they exist
			$$
			\lim_{m}\lim_n\epsilon(g_n)R_{h_m}f=\lim_{m}\lim_nf(g_nh_m)= \lim_{n}\lim_mf(g_nh_m)= \lim_{n}\lim_{m}\epsilon(g_n)R_{h_m}f
			$$
			if and only if $RO(f)$ is relatively weakly compact in $l_\infty(G)$, i.e. $f\in WAP(G_d)$. However, if $f\in C_b(G)$, this is clearly equivalent to $f\in WAP(G)$ by \Cref{loc21} and Grothendieck's theorem.
		\end{proof}
		
		\bigskip
		
		\begin{thm} \label{loc6} Let $G$ be a locally compact Hausdorff right topological group. The following hold:
			\begin{enumerate}
				\item $C_b(G,\si)\subset LC(G)$ and $ C_0(G,\si)\subset LC_0(G)$.
				\item $LC(G)$ separates points from closed sets in $G$ if and only if $G$ is topological.
				\item If $G$ is admissible,  $WAP(G)\subset LC(G)$ and $WAP(G)\cap C_0(G)\subset LC_0(G)$
			\end{enumerate}
			If further $G$ is $\si$-compact admissible or compact, then
			\begin{enumerate}[resume]
				\item $LC_0(G)= C_0(G,\si)$
				\item If $G$ is non-compact, $AP(G)\cap C_0(G)= \{0\}$ so that $AP(G)\oplus C_0(G,\si)\subset WAPG(G)$
			\end{enumerate}
		\end{thm}
		
		\begin{proof}
			We first prove 1 and 4. Since $(G,\si)$ is a semitopological group, it is clear that $C_b(G,\si) \subset LC(G)$ and that $C_0(G,\si)\subset LC_0(G)$.
			Suppose in addition that  $G$ is $\si$-compact admissible or compact. Since $LC_0(G)$ is a non-trivial closed, translation-invariant subalgebra of $G$, by \Cref{loc2} and \Cref{sisi}, $\text{Fix}(LC_0(G))$ contains $N(G)=\overline{\{e\}}^{\si\si}$. On the other hand, $\text{Fix}(LC_0(G))\subset \text{Fix}(C_0(G,\si)) = N(G)$. It follows that $\text{Fix}(LC_0(G))=N(G)$, so that by \Cref{loc3} and 2 of \Cref{loc2}, $C_0(G,\si)= LC_0(G)$.
			%
			
			If $LC(G)$ separates points from closed sets, the initial topology of $LC(G)$ coincides with the original topology on $G$ (see 8.15 in \cite{willard}). However, for all $f\in LC(G)$, $y_{\al}\to y$ in $G$ implies $f(xy_{\al})\to f(xy)$, since $L_xf \in C(G)$ for all $x\in G$. Thus, $G$ is a locally compact semitopological group, and by \cite{ellis}, a topological group. The converse is clear as $LC(G)= C_b(G)$ in the topological case. 
			
			To prove 3, consider $f\in WAP(G)$; by assumption, $RO(f)$ is relatively weakly compact in $C_b(G)$. By Grothendieck's double limit theorem (see Theorem A.5 in \cite{berglund})  and \Cref{loc21}, for any sequences $\{g_n\}$,$\{h_n\}\subset G$, the following middle limits are equal when they exist
			$$\lim_{m}\lim_n \epsilon(h_m)L_{g_n}f =\lim_{m}\lim_n\epsilon(g_n)R_{h_m}f= \lim_{n}\lim_{m}\epsilon(g_n)R_{h_m}f
			=\lim_{n}\lim_m \epsilon(h_m)L_{g_n}f $$
			
			Consider the set $\{L_gf\mid g\in \Ld(G)\} \subset C_b(G)$. Then, the above equality certainly holds for $\{g_n\}\subset \Ld(G)$, $\{h_n\}\subset G$ when the limits exist.
			Therefore, by Grothendieck's theorem again, $\{L_gf\mid g\in \Ld(G)\}$ is relatively weakly compact in $C_b(G)$. 
			For any $g\in G$ now, by admissibility there exists some net $\{g_{\al}\}\subset \Ld(G)$ such that $g_{\al}\to g$, and thus $L_{g_\al}f\to L_{g}f$ pointwise. However, by weak compactness, $\{L_{g_{\al}}f\}$ has a weak cluster point in $C_b(G)$, which must coincide with $L_gf$. Thus, $L_g{f}\in C_b(G)$ for all $g\in{G}$. If in addition, $f\in C_0(G)$, then, $\{L_{g_{\al}}f\}\subset C_0(G)$ so that its weak limit $L_g f \in C_0(G)$ 
			, whence $f\in LC_0(G)$.
			
			Lastly, if $G$ is non-compact, $\si$-compact admissible,  if $f\in AP(G)\cap C_0(G)$, then by 3,4, and \Cref{loc3}, $f\in C_0(G,\si)= C_0(G/N(G))$. If $N(G)=G$, the claim is now obvious; otherwise by \Cref{loc5} $G/N(G)$ is a non-trivial locally compact topological group and the claim follows from the classical result (see 4.2.2a of \cite{berglund}).
		\end{proof}
		
		Let us denote $G$ with the discrete topology by $G_d$. The unitary representations on $G_d$ correspond one-to-one with the non-degenerate representations on $l_1(G)$, which we denote by $\Sigma$. $\Sigma$ then induces a norm on $l_1(G)$, namely $\|f\|_* = \sup_{\pi\in \Sigma} \|\pi(f)\|$ for $f\in l_1(G)$. The completion of $l_1(G)$ under this norm is then a C*-algebra, denoted by $C^*(G_d)$ and known as the group C*-algebra of $G_d$. 
		
		\noindent Consider the positive-definite functions on $G$, i.e. functions $f: G\to \mathbb{C}$ satisfying $\sum_{i,j=1}^n c_i\bar{c_j}f(x_ix_j^{-1}) \geq 0$ for all $\{c_i\}_{i=1}^n\subset \mathbb{C}$ and $\{x_i\}_{i=1}^n\subset G$.  The span of $P(G)$ can be identified as the dual of $C^*(G_d)$, and under this dual norm, forms a Banach algebra, known as the {\it Fourier-Stieltjes algebra}, $B(G_d)$ on the discrete group $G_d$. In the right topological group setting the topological analogue has to be defined more carefully than the classical case, due to a lack of continuous representations. For theory on these algebras in the classical case, see \cite{kaniuth}.
		 
		Following \cite{lauloy1}, we define the {\it Fourier-Stieltjes algebra} of $(G,\tau)$, $B(G)$ to given by $B(G_d)\cap C_b(G)$. We further define the {\it Fourier algebra} on $G$, $A(G)$, to be the closure of $B(G_d)\cap C_c(G)$ in $B(G)$. The {\it $\si$- Fourier-Steiltjes and -Fourier algebras} on $G$, will then respectively $B(G,\si)=B(G)\cap C(G,\si)$ and $A(G,\si)=A(G)\cap C(G,\si)$.
		
		Lau and Loy have done some extensive work over the Fourier-Steiltjes algebra in the compact setting in \cite{lauloy2}. In particular, they showed the following characterization;
		
		\begin{thm}[Lau,Loy]\label{bg}
			If G be is an  admissible compact Hausdorff right topological group, then
			$$B(G)=B(G,\si)\cong B(G/N(G))$$
			\end{thm}
		
		Unlike the locally compact topological group setting where $B(G)$ completely identifies $G$ up to topological isomorphism, the above result indicates that this is not the case in our setting. It however does follow that the quotient group  $G/N(G)$ is uniquely identified by $B(G)$ up to topological isomorphism. In the locally compact setting we are unable to produce an analogue of this strong result and instead have the following;

		\begin{co} \label{loc7} Let $G$ be an admissible $\si$-locally compact right topological group.Then, $B(G)$, $A(G)$ are commutative Banach algebras satisfying the following:
			\begin{enumerate}
				\item $B(G)$ is closed under translations on $G$
				\item $B(G,\si) = B(G/N(G))$ and $A(G)= A(G,\si)$
				\item $B(G)$ separates points from closed sets if and only if $G$ is topological
				\item $A(G)$ separates points in $G$ if and only if $G$ is topological
			\end{enumerate}
		\end{co}
		
		\begin{proof} That $B(G)$ is closed in norm in $B(G_d)$ follows from the fact that the norm on latter bounds the uniform norm. 
	     Observe that $B(G)= B(G_d)\cap C_b(G)\subset WAP(G_d)\cap C_b(G)=WAP(G)$, where the first containment is a classical result (see \cite{burckel}), and the equality follows from \Cref{loc211}. By 3 of \Cref{loc6}, $B(G)\subset LC(G)$ so that invariance under translations follows by 4.2.3 in \cite{berglund}. Now using \Cref{loc3}, $B(G,\si) = B(G_d)\cap C(G,\si)\subset B(G_d)\cap C(G/N(G))=  B(G/N(G))$. The other inclusion follows by \Cref{loc3} or \Cref{ng}.
			
		For 4, we again obtain the containments $B(G_d)\cap C_c(G)\subset WAP(G)\cap C_c(G)\subset LC_0(G) = C(G,\si)$, where we used 3, 4 of \Cref{loc6} and \Cref{loc3}. It follows that the closure of $B(G_d)\cap C_c(G)$ in $B(G)$, namely $A(G)$, is contained in $B(G,\si)$, whence $A(G)= A(G,\si)$. The last two claims follows from 3 of \Cref{loc6}.
		\end{proof}
		Unlike \Cref{bg}, where $B(G)\cong B(H)$ for a topological group $H$, it is possible that there exist non-compact locally compact admissible, or compact non-admissible right topological groups for which this property never holds. We wonder if such groups possess any special properties. \\
\bigskip
		
		We conclude the section with a short result on representations generalized from \cite{lauloy1}.
		\begin{co}
			Let $G$ be a $\si$-locally compact admissible right topological group. Then $A(G)^*$ is a Von Neumann algebra corresponding to a faithful representation $\pi$ of $G$, with $\la{\pi(x)}{f} = f(x)$ for all $x\in G$, $f\in A(G)$ if and only if $G$ is topological.  
		\end{co}
		
		\begin{proof} Suppose $A(G)^*$ is a Von Neumann algebra as described. By \Cref{loc7}, $A(G)= A(G,\si)$, so that $x\mapsto \la{\pi(x)}{f} = f(x)$ is $\si$-continuous and factors through $N(G)$ by \Cref{loc3}. Since $VN_\pi$ is determined by its evaluation on $A(G)$, it follows that $\pi$ is $N(G)$-invariant. This contradicts faithfulness unless $G$ is topological. The converse is a standard result (see \cite{kaniuth}).
		\end{proof}
		%
		%
		%
		\bigskip
		\section{Measure Algebras} \label{measuresec}
		
		In this section we present measure algebra analogues for right topological groups and present their properties. We then use these to characterize the existence of a Haar measure on compact right topological groups. In the process we give some alternate conditions that do not rely on admissibility. \\
		
		\noindent Suppose $\mu,\nu\in M(G)$; note that for $f\in \cc(G)$, the usual convolution of measures, $\int\int f(xy)d\mu(x)d\nu(y)$, may not be defined, since the function $f\cd \mu$ defined by $G \to \mathbb{C}$, $y\mapsto \int f(xy) d\mu(x)$ is not necessarily Borel (let alone continuous like in the topological case). As such, following \cite{lauloy1}, let us define
		$$\mc=\{\mu\in M(G)\mid f\cd\mu\in C_b(G),\text{ for all }f\in C_b(G)\}$$
		
		Throughout, for each $g\in G$, we shall denote by $\mu_g$, $_g\mu$ the left and right translation of $\mu$ by $g$, i.e $\mu\circ R_g$ and $\mu\circ L_g$ respectively. We similarly denote by $\mu_G$ the right orbit of $\mu$.\\
		
		Lau and Loy's result on $\mc$ for a compact right topological group $G$ is true for the locally compact case as well, presented as follows.

		\begin{pr} Let $G$ be a locally compact Hausdorff right topological group. Then, $(\mc,\Box)$ is a Banach algebra that is closed in $M(G)$. Furthermore, $l_1(\Ld(G))\subset \mc$, so that the space is non-trivial.
		\end{pr}
		
		\bigskip
		
		We may then prove the following important divergence from the classical case.

		\begin{thm} \label{cont}
			If $G$ is a locally compact Hausdorff right topological group, then the following are equivalent:
			\begin{enumerate}
				\item The right regular representation of $G$ is continuous
				\item $C_c(G)\subset \mc$ 
				\item $G$ is topological. 
			\end{enumerate}
		\end{thm}
		
		\begin{proof}
			Suppose the right regular representation is continuous. Since the inverse map on the unitaries is WOT-WOT continuous, it follows that $G\to B(L_2(G))$, $x\mapsto R_{x^{-1}}$ is WOT continuous, whence SOT continuous by equivalence of the two topologies on the unitaries. Suppose $f\in C_c(G)$ with support $K\subset G$, and $h\in l_{\infty}(G)$; then, if $y_{\al} \to y$ in $G$,
			\begin{align*}\left|\int h(x) \ (f(xy_{\al}^{-1})- f(xy^{-1})) \ d\ld(x)\right| 
			&\leq {\|h\|}_{\infty}\  \int \ci_{K{y_{\al}}\cup Ky} \ |f(xy_{\al}^{-1})- f(xy^{-1})| \ d\ld(x) \\
			&\leq {\|h\|}_\infty \ {\|\ci_{K{y_{\al}}\cup Ky}\|}_2 \ {\|R_{y_{\al}^{-1}}f- R_{y^{-1}}f\|}_2\\
			& \leq 2{\|h\|}_{\infty} \ \ld(K) \ {\|R_{y_{\al}^{-1}}f- R_{y^{-1}}f\|}_2 \  \longrightarrow 0
			\end{align*}
		by SOT-continuity of the right regular representation of $G$. It follows that $h\cd fd\ld\in C_b(G)$, i.e. $\cc(G)\subset \mc$.

			Suppose now that 2 holds. We claim that $\cc(G)\subset \mc$ causes $LC(G)$ to separate points from closed sets. Consider any closed set $M$ such that $e\not\in M$. Since $G\bs M$ is open, there exist strict inclusions $$e\in K\subset V\subset \overline{V} \subset U \subset \overline{U}\subset G\bs M$$ 
			where $V$, $U$ are open neighbourhoods of $e$, and $K$, $\overline{V}$, $\overline{U}$ are compact neighbourhoods of $e$. Since $G$ is locally compact Hausdorff, by Urysohn's lemma, there exist continuous functions $f,p: G\to [0,1]$ such that $\ci_K\leq f\leq \ci_V$, and $\ci_{\overline{V}}\leq p\leq  \ci_{U}$ so that $h=1-p$ satisfies $\ci_{U^c}\leq h\leq \ci_{\overline{V}^c}$. Note here that $\overline{V}^c \supset U^c\supset \overline{U}^c\supset M$, so that $\overline{V}^c$ is a neighbourhood of $M$.
			
			Since $f\in C_c(G)$, the map $\phi: y\mapsto \int h(xy)f(x)d\ld(x)$ is continuous. Furthermore,
			
			$$\int h(xy)f(x)d\ld(x)=\int_{V\cap \overline{V}^cy^{-1}} h(xy)f(x)d\ld(x)\leq \|f\|_{\infty}\|h\|_{\infty} \ld(V\cap \overline{V}^cy^{-1})$$
			so that $\phi$ is non-zero at $y$ implies that $V\cap \overline{V}^cy^{-1} \neq \varnothing$. However, $x\in V\cap \overline{V}^cy^{-1}\implies xy= vy\in \overline{V}^c$ for some $v\in V$, which implies that $y\in V^{-1}\overline{V}^c$, i.e. $\supp(\phi) \subset V\cap \overline{V}^cy^{-1}$. Since $\overline{V}^{c}\cap V \subset V^c\cap V = \varnothing$, it follows that $e\not\in V^{-1}\overline{V}^c$, i.e. $\phi(e)= 0$. On the other hand, if $y\in \overline{U}^c$, $V\cap \overline{V}^cy^{-1} \supset K\cap U^cy^{-1}$ is a neighbourhood of $e$, and thus, $\ld(V\cap \overline{V}^cy^{-1})>0$, so that $\phi(e)>0$. In fact, for $y\in \overline{U}^c$
			\begin{align*}
			\int h(xy)f(x)d\ld(x) =\int_{V\cap \overline{V}^cy^{-1}} h(xy)f(x)d\ld(x) &\geq \int_{K\cap \overline{U}^c y^{-1}} h(xy)f(x)d\ld(x) \\
			&= \ld({K\cap \overline{U}^c y^{-1}})>0
			\end{align*}
			Suppose, for some $\ve>0$, $\int h(xy)f(x)d\ld(x)>\ve$ for all $y \in \overline{U}^c$, whence for all $y\in M$. Then, it is clear that $\phi|_M>\ve$, $\phi(e)=0$. Thus, $\phi$ separates $e$ from $M$.  In other words, the set $D=\{ h\cd f d\ld \mid h\in C_b(G), f\in C_c(G)\}\subset C_b(G)$ separates points from closed sets in $G$. Note however, that 
			{\allowdisplaybreaks \begin{align*}
				L_x[h\cd fd\ld ](z)=[h\cd fd\ld ](xz) = \int h(yxz)f(y)\ld(y)&= \int h(y)f(yz^{-1}x^{-1})\ld(y)\\
				&=\int h(y)R_{x^{-1}}f(yz^{-1})d\ld(y)
				\end{align*}
			}
			is continuous in $z$ since $R_{x^{-1}}f \in C_c(G)$ still holds. In other words $D\subset LC(G)$, is a set that separates points from closed sets on $G$. By \Cref{loc6}, $G$ is a topological group.
			
			Lastly 3 implies 1 is a standard result (see \cite{folland}).
		\end{proof}

		\bigskip
		
		In the compact admissible case, it was shown in \cite{milnes1} that no continuous representation of $G$ may be faithful unless $G$ is topological. This was used to prove \Cref{cont} in \cite{lauloy1} for compact groups. We cannot rely on this in our more general setting, and our proof is constructive. Whether faithful representations exist for general locally compact admissible right topological groups remains an open question.
		
		\begin{co}
			Either one of $\lgg= L_1(G)$, $\lc= L_1(G)$ or $\mc= M(G)$ holds if and only if $G$ is topological.
		\end{co} 
		
		\begin{proof}
			Since $\lgg,\lc\subset \mc$, and $\cc(G)\subset L_1(G)\subset M(G)$, by \Cref{cont} the forward implication holds. The converse follows from $LC_0(G)$ being contained in the uniformly continuous functions on $G$ when it is topological \cite{folland}.
		\end{proof}
		
		\bigskip
		
		The following two results presented in \cite{lauloy1} generalize to locally compact right topological groups without much effort and we omit the proof.
		
		\begin{pr}
			Let $G$ be a locally compact Hausdorff right topological group. Then, $\lgg$ is a closed right ideal in $\mc$ containing all $\mu\in L_1(G)$ such that $x\mapsto |\mu|(Kx^{-1})$ is continuous for all $K\subset G$ compact.
		\end{pr}
	
	\bigskip
	
		%
		%
		
		\bigskip

		\begin{pr}
			Let $G$ be a locally compact Hausdorff right topological group. Then, $\lc$ is a closed left ideal in $\mc$ that is an L-space. Moreover, for all all $\mu\in \lc$, $y\mapsto \mu(Ey)$ is continuous.
		\end{pr}
		%
		%
		%
		\bigskip
		\bigskip
		
		\noindent Let us denote by $$D(G)= \{f\in C_b(G)\mid g\mapsto L_yf(g^{-1}) \text{ is continuous for all } y\in G\},$$ as in \cite{lauloy1}. 
		
		In general, the relation between $\lgg$ and $\lc$ is unknown. However, the following proposition gives some common elements.
		
		\begin{pr} If $G$ is locally compact then, $L_1(G)\cap WAP(G)\cap D(G)\subset\lgg$; in particular,  $L_1(G)\cap AP(G)\cap D(G) \subset \lgg\cap \lc$. 
		\end{pr}
		
		\begin{proof} 
		Suppose $f\in  L_1(G)\cap WAP(G)\cap D(G)$. For each $x\in G$, $y\mapsto  f(xy^{-1}) = R_{y^{-1}}f(x)$ is continuous. If $y_{\al}\to y$ in $G$, then, $f\in WAP(G)$ implies that every subnet of $\{R_{y^{-1}_{\al}}f\}$ has a weak cluster point in $C_b(G)$ which, by pointwise continuity coincides with $R_{y^{-1}}f$. It follows that every subnet of $\{R_{y^{-1}_{\al}}f\}$ has a subnet that converges to $R_{y^{-1}}f$, whence $R_{y^{-1}_{\al}}f\to R_{y^{-1}}f$ weakly. Thus, $y\mapsto R_{y^{-1}}f$ is weakly continuous and $f\in \lgg$.	A similar proof shows that for $f\in  L_1(G)\cap AP(G)\cap D(G)$, $y\mapsto R_{y^{-1}}f$ is uniformly continuous and the latter claim follows.
		\end{proof}
		\bigskip
		\bigskip
		
		\subsection{Algebras characterizing the Haar measure}
		
		 Throughout this subsection, we assume that all groups are compact. For every measure $\mu\in M(G)$, we denote by $\mu_G$, the right orbit of $\mu$ under $G$. Then, for a compact Hausdorff right topological group, we define,
			$$\mw = \{\mu\in \mc \mid \mu_{G} \text{ is relatively weakly compact}\}$$
		Note that $\lc\subset \mw\subset \mc$.\\
		
		\begin{pr} \label{first}
			$\mw$ is a closed, left ideal of $\mc$ containing $\lc$ and left  and right invariant under $\Ld(G)$.
		\end{pr}
		
		\begin{proof}
			It is obvious that $\mw$ is a subspace of $\mc$. Suppose $\{\mu_n\}\subset \mw$, and $\mu_n\to \mu$. We shall show that $\mu_G$ is relatively weakly compact using Grothendieck's double limit theorem. Suppose for a bounded sequence $\{T_n\}\subset M(G)^*$, $\lim_n\lim_m T_n(\mu_{g_m})$ and $\lim_n\lim_m T_m(\mu_{g_n})$ exist.	
			
			For $T\in M(G)^*$ and $g\in G$, $T\circ R_g^* \in M(G)^*$ and $$T(\mu_g) - T({(\mu_k)}_{g}) =  TR_{g}^*(\mu-\mu_k) \leq \|T\|\|\mu-\mu_k\|\to 0$$ uniformly over $g\in G$, and $T$ contained in a normed bounded subset of $M(G)^*$. Using this, we have, 
			\begin{align*}
			\lim_n\lim_m T_nR_{g_m}^*(\mu) =   \lim_n\lim_m [T_nR_{g_m}^*](\lim_k\mu_k)
			= \lim_k \lim_n\lim_m  [T_nR_{g_m}^*](\mu_k) 
			&{=^{(\ast)}} \lim_k \lim_m\lim_n [T_nR_{g_m}^*](\mu_k)\\
			&= \lim_m\lim_n \lim_k  [T_nR_{g_m}^*](\mu_k)\\
			&= \lim_m\lim_n   [T_nR_{g_m}^*](\mu)
			\end{align*}
			where we used at $(\ast)$ the fact that for an appropriately selected subsequence in $k$, $\lim_n\lim_m  [T_nR_{g_m}^*](\mu_k)$ and $\lim_m\lim_n [T_nR_{g_m}^*](\mu_k)$ exist (since the limit of each iterated limit exists as $k\to \infty$), and ${[\mu_k]}_G$ is weakly compact for all $k\in \mathbb{N}$. The claim thus follows.
			
			Suppose now that $\mu\in\mc$, $\nu\in \mw$, and $K\subset G$ is an arbitrary non-empty set so that for each $x\in K$, $T\in M(G)^*$, $\la{{[\mu\Box\nu]}_{x^{-1}}}{T} = \la{\mu\Box{\nu_{x^{-1}}}}{T}$. 
			Since $\nu_G$ is relatively weakly compact, $\{\nu_{x^{-1}}\mid x\in K\}$  has a weak cluster point, say $\gamma\in M(G)$. Then, $\mu\Box\gamma$ is clearly a weak cluster point of $\{{[\mu\Box\nu]}_{x^{-1}}\mid x\in K\}$.  It follows that ${[\mu\Box\nu]}_{G}$ is relatively weakly compact and $\mu\Box\nu\in \mw$.
			
			That $\lc$ is contained in $\mw$ is easy to see - for each $\mu\in \lc$, $x\mapsto \mu_{x^{-1}}$ is norm continuous, thus weakly continuous, so that the image of the map, $\mu_G$ is weakly compact. 	
			
			As $\mc$ is closed under right translation by $\Ld(G)$, and this does not change the right orbit of an element of $\mc$, $\mw$ is closed under right translation. Further, if $\mu\in \mw$, and $T\in M(G)^*$,  for each $g\in \Ld(G)$, $T(_{g^{-1}}\mu) = T\circ L_{g}^* \in M(G)^*$. Thus, the right orbit of $_{g^{-1}}\mu$ has a weak cluster point via the relative weak compactness of $\mu_{G}$ and $_{g^{-1}}\mu\in G$.  
		\end{proof}
		
		\bigskip
		
		\begin{pr} The following are equivalent:
			\begin{enumerate}
				\item $G$ has a Haar measure
				\item $\lc$ contains a non-zero measure
				\item $\mw$ contains a non-zero measure. 
			\end{enumerate}	
		\end{pr}
		
		\begin{proof}
			That 1 implies 2 and 2 implies 3 is trivial. We shall prove that 3 implies 1.
			
			Suppose $\mu\in \mw$ is non-zero. Then, by scaling, we may assume $\mu(G) =1$. By assumption the weak closure of $\mu_G$, is compact so that by Krein Smulian, its closed convex hull, $K$, is also weakly compact. It is clear that for all $\nu\in K$, $\nu(G)=1$.
			
			We claim that $K$ is closed under right translations. Indeed, suppose $\nu\in K$, $h\in G$ and let $\ve>0$.  For each $T\in M(G)^*$, there exists some convex combination $\sum_{i=1}^n \al_i\mu_{g_i}$ such that $|\la{TR_h^*}{\nu -\sum_{i=1}^n \al_i\mu_{g_i}}|=|\la{T}{ \nu_{h^{-1}} -\sum_{i=1}^n \al_i(\mu_{g_i})_{h^{-1}}}| = |\la{T}{ \nu_{h^{-1}} -\sum_{i=1}^n \al_i\mu_{h^{-1}g_i}}|<\ve$. As $\ve$ was arbitrary, it follows $\nu_h\in K$. Now $G$ acts on $M(G)$ isometrically i.e. for each $g\in G$, $\mu\in M(G)$, $\mu\mapsto \mu_{g^{-1}}$ is a continuous isometry. By Ryll-Nardzewski fixed point theorem, $K$ has a fixed point of $G$. This is clearly a Haar measure on $G$.
		\end{proof}

		\bigskip
		 An issue with the previous proposition is that we are unsure apriori of what elements are contained in $\mw$. We do however have $C(G,\si)d\ld\subset \lc$ when $G$ has a Haar measure $\ld$.
		
		%
		
		\bigskip
		\bigskip
		
		We now define, $\msc = \{\mu\in M(G)\mid f\cd \mu \in C(G,\si)\text{ for each }f\in C(G)\}$. 
		Clearly $\msc\subset \mc$. If $G$ has a Haar measure, let us set $\mathcal{L}_\si(G) = \mathcal{M}_\si(G)\cap L_1(G)$.\\

		\begin{pr} \label{loc23}
			$\mathcal{M}_\si(G)$ is a closed left-ideal of $\mc$ containing the Haar measure on $G$ (if it exists). Furthermore, $\mathcal{M}_\si(G)$ and $\mathcal{L}_\si(G)$ are closed under right translations.
		\end{pr}
		
		\begin{proof}
			Let us first show that $\msc$ is closed. Indeed, suppose $\{\mu_\al\}_{\al\in A}\subset \msc$ such that $\mu_\al \to \mu$, $\mu\in \mc$. Then, for any $f\in C(G)$, $\{x_\be\}_{\be\in B}\subset G$ such that $x_\be \to^{\si} x\in G$,
			\begin{align*} 
			&|f\cd \mu(x_\be) - f\cd \mu(x)|\\
			&= \left |\int f(yx_\be)- f(yx) d\mu(y)\right |\\
			&\leq \left |\int f(yx_\be) d(\mu-\mu_\al)(y) \right|+ \left|\int  f(yx_\be)- f(yx) d\mu_\al(y)\right|
			+ \left|\int f(yx) d(\mu-\mu_\al)(y) \right|\\
			&\leq 2{\|f\|}_\infty \|\mu-\mu_\al\| + |f\cd \mu_\al(x) - f\cd \mu_\al(x_\be) |
			\end{align*}
			Fix $\ve>0$. Then, choosing an appropriate $\al\in A$ so that the former term is less than $\ve/2$ and an appropriate $\be$ for $\al$ so that the latter term is less than $\ve/2$ (as $f\cd \mu_\al$ is in $C(G,\si)$), we have a bound of $\ve$ above which gives us our claim.
			
			To see that $\msc$ is a left ideal, suppose that $\mu\in \mc$ and $\nu\in \msc$. For any $f\in C(G)$, we have,
			$$	[\mu\od \nu] \cd f(x) = \la{\mu\od \nu}{ R_xf} = \la{\nu}{R_xf\cd \mu}
			= \la{\nu}{R_x (f\cd \mu)}
			= [f\cd \mu]\cd \nu (x)$$
			Since $f\cd \mu \in C(G)$, it follows that $[\mu\od \nu] \cd f \in C(G,\si)$. This proves our second claim. Note that if $\ld$ is a Haar measure on $G$, then for any $f\in C(G)$,  $f\cd \ld$ is always a constant and thus, trivially in $C(G,\si)$.
			
			We  show the last claim for $\msc$ and the claim for $\mathcal{L}_\si(G)$ will simply follow by absolute continuity of translations of its elements. For $\mu\in \msc$, $x, g\in G$, and $f\in C(G)$,
			one easily checks $f\cd \mu_g = R_g^{-1} [f\cd\mu]$. Now as $f\cd \mu \in C(G,\si)$ and $(G,\si)$ is a semitopological  group, it follows that $R_g^{-1} [f\cd\mu]\in C(G,\si)$ which concludes our claim.
		\end{proof}
		
		\bigskip
		
		
		In \cite{lauloy1}, it was shown that, for admissible groups, $WAP(G) = AP(G) = C(G,\si)$. Using this, we have the following result.
		
		\bigskip
		\begin{pr}\label{heheh}
			If $G$ is admissible, then $\mw\subset \msc$.
		\end{pr}
		
		\begin{proof}
			By Proposition 2.3 in \cite{lauloy1}, $WAP(G) = AP(G) = C(G,\si)$.
			
			Suppose $\mu\in \mw$, $f\in C(G)$, then for all $y\in G$, 
				$${f\cd\mu_{g^{-1}}}(y) =  \int f(xgy) d\mu(x)= L_g[ f\cd \mu](y) $$
				Moreover, f\\
				or every $T\in l_\infty(G)^*$, the map $\mc\to \mathbb{C}$, $\nu\mapsto \la{T}{f\cd\nu}$ is norm continuous i.e. lies in $\mc^*$. Then, since $\mu_G$ is relatively weakly compact, $\{\mu_{g^{-1}}\mid g\in G\}$ has a weak cluster point, giving a weak cluster point for $\{f\cd\mu_{g^{-1}}\mid g\in G\}=\{L_g[f\cd\mu]\mid g\in G\}\subset l_\infty(G)$.
			 It follows that $\{L_g[f\cd\mu]\mid g\in G\}\subset l_\infty(G)$ is relatively weakly compact i.e.   $f\cd\mu \in WAP(G_d)$ (the left and right weakly almost periodic functions coincide on a topological group, see \cite{berglund}). As $WAP(G) = WAP(G_d)\cap C(G)$ (\Cref{loc211}),  $f\cd\mu\in C(G,\si)$, i.e. $\mu\in \msc$.
		\end{proof}
		
		\bigskip
		
		Consider the  flow $(\rho_G, G)$, where $\rho_G$ denotes the right translation action of $G$ on itself. An equicontinuous right topological group is defined to be a right topological group $G$, satisfying the property that it's right flow is equicontinuous i.e. all the maps $\rho_G(g)$, $g\in G$ are equicontinuous. Milnes showed that $G$ is equicontinuous if and only if $AP(G) = WAP(G)= C(G)$ (see \cite{milnes2}). Thus, for such groups the previous result looks very different - and \Cref{first} is not redundant to the following result. See \Cref{hehe}.
		
		\bigskip
		\begin{thm}\label{msc}
			$G$ has a right invariant Haar measure if and only if $\msc$ contains a non-zero measure. 
		\end{thm}
		
		\begin{proof}
			Suppose $\mu\in \msc$ is non-zero. Then, by scaling it, we may assume that $\mu(G) = 1$. Consider $\mu_G=\{\mu_g\mid g\in G\}$. Since $g\mapsto \mu_g$ is w*-continuous, $\mu_G$ is compact and contained in $\msc$ by \Cref{loc23}. Consider $C= \cch \mu_G$, the weak*-closed convex hull of $\mu_G$. Since the set is bounded by $\|\mu\|$ and is w*-closed, by the Banach-Alaogu theorem, it is compact. Moreover $C$ is preserved under right translations $\{R_g\mid g\in G\}$, which are w*-w*-continuous affine maps on it. 
			
			 We claim that right translates by $G$ are also distal on $K$. Indeed, suppose $\lim_{\al\in A} \mu_{g_\al}= \lim_{\al\in A} \nu_{g_\al}$ for some $\mu,\nu\in K$ and some $\{g_\al\}_{\al\in A}\subset G$. Since $G$ is compact, $\{g_\al^{-1}\}_{\al\in A}$ has some subnet $\{g_\be^{-1}\}_{\be\in A}$ that converges to some $g^{-1} \in G$. Then, $\mu_{g_\be}\to \mu_g$ and $\nu_{g_{\be}}\to \nu_{g}$, so that $\mu_g=\nu_g$ and thus $\mu=\nu$.
			 
			
			By Namioka's fixed point theorem then, $G$ has a fixed point in the w*-closed convex hull of $K$. Since $\mu(G)=1$, all the elements $\nu$ of $C$ satisfy $\nu(1)=1$. It therefore follows that $G$ has a right invariant Haar measure.
		\end{proof}
		\bigskip


		\begin{pr}
			Let $G$ be a compact hausdorff right topological group. For $\mu\in \mc$, if  $R_G\mu\in \mc$, then $\mu\in \mathcal{M}_\si(G)$. 
		\end{pr}
		
		\begin{proof}
			Suppose $\mu\in \mc$ is such a measure. Then, one has, $f\cd \mu_{x}$ is continuous for all $x\in G$. However, $f\cd \mu_{x}=L_{x^{-1}}f\cd\mu$. It thus follows that $f\cd\mu\in C_c(G) = C(G/N(G))$, so that $\mu_{x}$ is $N(G)$-invariant and in $\mathcal{M}_\si(G)$. 
		\end{proof}
		\bigskip
		
		Let $H$ be a closed subgroup of $G$ with a right Haar measure $\mu$. We then denote by $\ld_H$, the measure on $G$ given by $f\mapsto \int f d\ld_H = \int_{H} f|_{H} (x) d\ld(x)$. Observe that $\ld_H$ is always right $H$-invariant and is left $H$-invariant if $H$ is compact topological.  
				
				\begin{thm}\label{subgroup}
					Suppose $G$ is a $\si$-locally compact admissible or compact Hausdorff right topological group and $H\subset  \Ld(G)$ is a normal compact topological subgroup. Then, $\ld_{H}\in\mc$. In the compact case, if $H$ is $\si\si$-closed, then $\ld_H\in \msc$ and $G$ has a Haar measure.
				\end{thm}

				\begin{proof}
					 As multiplication is separately continuous on $H$, by Ellis' theorem, $H$ is a compact topological group. Therefore, $H$ has a unique invariant Haar measure and we may define $\ld_H$ as above.
					
					We will show that $\ld_{H}\in \mc$. Firstly, note that  $H\times G\to G$, $(h,g) \mapsto hg$ is separately continuous.  By Namioka's theorem, it is jointly continuous so that for any $f\in C(G)$,  $H\times G\to \mathbb{C}$, $(h,g) \mapsto f(hg)$ is jointly continuous. We claim that $\{L_hf\mid h\in H\}\subset C(G)$ are equicontinuous on $G$.
					
					As $(h,g)\mapsto f(hg)$ is continuous at every $(m,g_0)\in H\times G$, there exist neighbourhoods $U_m\times V_m\subset H\times G$ of $(m,g_0)$, such that for all $(h,g)\in U_m\times V_m$, $|f(hg)-f(mg_0)|<\ve$. As $\{U_m\}_{m\in H}$ is an open cover for $H$, there exists a finite subcover $\{U_{m_i}\}_{i=1}^k$. Consider $V = \cap_{i=1}^n V_{m_i}$, also a neighbourhood of $g_0$. Then, for any $m\in H$, $m\in U_{m_i}$ for some $1\leq i\leq n$. Thus, for every $g\in V$, $|f(mg) - f(mg_0)|<\ve$, and as this holds for every $m\in H$,  $\sup_{m\in H}|f(mg)-f(mg_0)|<\ve$. It follows that, $\sup_{m\in H}|f(mg)-f(mg_0)| \to  0$ as $g\to g_0$. This proves the claim.
					
					Now as $\{L_nf\}_{n\in H}$ are equicontinuous, note that, $$f\cd \ld_H(y) = \int_{H} R_yf|_{H}(x) d\ld(x) = \int_{H} f|_{x\in H}(xy) d\ld(x)  = \int_{H} L_x f(y) d\ld(x) $$
					and $y_\al\to y$ implies $L_xf(y_\al)\to L_xf(y)$ for all $x\in H$, thus, 
					$|L_xf(y_\al)-L_xf(y)| \to 0$ for all $x\in H$, and $$ |f\cd \ld_H(y)- f\cd \ld_H(y_\al)| = \left|\int_{H} L_x f(y) d\ld(x) - \int_{H} L_x f(y_\al) d\ld(x)\right | \to 0$$  
					Thus, $f\cd \ld \in C(G)$ and $\ld \in \mc$.
					
					Now note that for $n\in H$ and $f\in C(G)$, \begin{align*}f\cd \ld(ny) = \int_{H}  R_{ny}f|_{H}(x)d\ld_{H}(x) =\int_{Hn}  R_{y}f|_{Hn}(x)d\ld_{H}(xn^{-1}) &=\int_{Hn}  R_{y}f|_{Hn}(x)d\ld_{H}(xn^{-1})\\ &= \int_{H}  R_{y}f|_{H}(x)d\ld(x)_{H}(x)\\ &= f\cd \ld(y)
					\end{align*}
					Thus, $f\cd \ld(y) \in C(G/H)$. 
					
					If $H$ is $\si\si$-closed, then, by \Cref{sisi}, $H\supset N(G)$ - thus for each $f\in C(G)$, $f\cd \ld_H \in C(G/N(G))$. As $C(G,\si)=C(G/N(G))$ by Lau and Loy, $\ld_H\in \msc$. A Haar measure then exists by \Cref{msc}.
				\end{proof}

		\bigskip
		
		\bigskip

		\begin{co}
			If $G$ is $\si$-locally compact admissible or compact, and has a Haar measure, for every compact subgroup $H\subset \Ld(G)$, it has a left-$H$ invariant right Haar measure. In particular, $G$ is compact and $\Ld(G)$ is closed, then, $G$ has a two-sided invariant Haar measure.
		\end{co}
		
		\begin{proof}
			By Ellis' theorem $H$ is a compact topological group and thus has a Haar measure. By \Cref{subgroup} then, $\ld_H\in \mc$. Given a Haar measure $\ld$ on $G$, $\ld_H \Box \ld$ then gives the desired measure.\end{proof}

		\bigskip
		\bigskip
		
		Analogously to the topological center, we define the Borel center of G as follows $$\Ld_b(G)=\{g\in G\mid x\mapsto gx \text{ is Borel}\}$$ It is clear that $\Ld(G)\subset \Ld_b(G)$. 
	
		When $G$ has a Borel multiplication map, we say $G$ is a Borel group (in this case $\Ld_b(G)=G$). If $G$ is compact and metrizable, then the condition $\Ld_b(G)=G$ coincides with $G$ being Borel (see 4.55 of \cite{aliprantis}). However, we cannot make use of this in the admissible case due to the following generalization of Namioka's result in \cite{namioka}; 
		
		\begin{pr} \label{metrizable}
			Suppose $(G,\tau)$ is a locally compact admissible Hausdorff metrizable right topological group. Then, $G$ is topological.
		\end{pr}
		
		Namioka's proof here generalizes without much effort to locally compact case.\\
		
		We attempt to use metrizability in a different way.
		\begin{pr} Let G be a $\si$-locally compact and admissible or compact, Hausdorff right topological group. If $G$ has a $\si\si$-closed, compact normal subgroup $H\subset \Ld_b(G)$ such that $H$ is metrizable and has a right invariant Haar measure, then $G$ has a Haar measure.
		\end{pr}
		
		\begin{proof}
			Suppose $H\subset \Ld_b(G)$ is as given. By \Cref{sisi}, $(G/H,\tau)$ is a locally compact topological group and thus has a Haar measure, say, $\ld_{G/H}$. Let $\ld_H$ be a Haar measure on $H$.
			
			Fix $f\in C(G)$. Note that the map $H\times G\to \mathbb{C}$, $(x,y)\mapsto f(xy)$ is continuous in the first variable, and Borel in the second variable Furthermore, $H$ being compact Hausdorff metrizable is separable. Thus,  the  map is jointly measurable  (see 4.51 of \cite{aliprantis}). Applying Fubini's theorem, the bounded function $y\mapsto F(y) = \int f|_H (xy)d\ld_H(x)$ is Borel measurable and clearly left $H$-invariant. We may thus integrate it with respect to $\ld_{G/H}$, and define $\ld\in M(G)$ in the standard way: $\int f(x)d\ld(x) = \int   \int f|_H (xy)d\ld_H(x) \ld_{G/H}(x)$. As $\ld_{G/H}$ is a $G$-invariant Haar measure, one easily notes that $\ld$ is a Haar measure on $G$.
		\end{proof}

		\bigskip
		\begin{re}
			The above proof still works if one assumes that $G/H$ has a Haar measure instead of $H$ being $\si\si$-closed.  For instance, if $H$ is a compact metrizable topological group, the hypothesis is met. While the hypothesis  on $H$ is strong, we note that since admissibility on $H$ is not assumed, \Cref{metrizable} does not come into play - therefore, $H$ may be non-trivially right topological.
		\end{re}
		\bigskip

		\section{Examples and open problems} \label{openprob}
		
We present some examples of $\si$-locally compact right topological groups. An excellent source of examples in the compact case is \cite{milnes1}, and our own examples have been inspired by this paper. Additionally \cite{milnes3} provides a general framework for constructing such examples via the concept of Schreier products of groups.
 		\begin{ex}
			Consider the group $\mathbb{C}_*\times \mathbb{C}^{AP}= \mathbb{C}_*\times \hat{\mathbb{C}_d}$. We equip this group with the following multiplication:
			$$(w,h)(v,g) = (wv, R_vh g)$$
			
			If $G$ is further equipped with the product topology, one obtains a $\si$-locally compact right topological group. Since $\mathbb{C}\hookrightarrow \hat{\mathbb{C_d}}$ via $z\mapsto [t\mapsto e^{2\pi izt}]$ sits inside $\ld (C^{AP})$, i.e is the set of continuous characters on $\mathbb{C}$, $\Ld(G) \cong\mathbb{C}\times \mathbb{C}$. Further one notes that $G$ admissible. 
			
			Consider the  normal subgroup $H=\{1\}\times \mathbb{C}^{AP}$. Taking the quotient of $G$ with respect to this subgroup, one obtains an algebraic isomorphism onto $\mathbb{C}\times {1}$, in fact, we get the following composition	
			$$\mathbb{C}^*\times \mathbb{C}^{AP} \to [\mathbb{C}^*\times \mathbb{C}^{AP}]/H \to \mathbb{C}$$
			$$(w,h)\mapsto [w,1] \to w$$
			
			Since the above diagram commutes, it is clear that $G/H\to \mathbb{C}$ is a continuous bijective homomorphism. By \Cref{openmap}, $G/H\cong \mathbb{C}$, whence by \Cref{sisi}, $H$ is $\si\si$-closed and thus contains $N(G)$. One notes however, that no proper subgroup of $N\subset H$ provides a topological quotient $G/N$ so that $H=N(G)$ as $N(G)=\overline{\{e\}}^{\si\si}$. The hypothesis of \Cref{haarmeas} is satisfied for $G$. Indeed, one gets a compact strong normal system of subgroups given by $G\supset N(G) \supset \{e\}$ as $G\times N(G)\to G$, $(w,h)(1,g)= (w, hg)$ is clearly separately continuous, whence jointly continuous by \Cref{separatecont}. A Haar measure on $G$ is simply given by the product of Haar measures on $\mathbb{C}$ and $\mathbb{C}^{AP}$ respectively. 
			
			Here we observe that $G'=\mathbb{T}\times \mathbb{C}^{AP}$ forms a subgroup of $G$ that is also an admissible right topological group (discussed in \cite{milnes1}) and further, $N(G')=N(G)$. An open question is determining when subgroups of a $\si$-locally compact right topological group satisfy the latter property.
		\end{ex}

		\bigskip
		
		\begin{ex} Let $G$ be the group $\mathbb{T}\times \mathbb{C}\times \mathbb{C}^{AP}=  \mathbb{T}\times \mathbb{C}\times \hat{\mathbb{C}_d}$ with multiplication given  by:
			$$ (u,x,h)(v,y,g)=(uvh(y), x+y, hg)$$
			
			Along with the product topology and the specified multiplication, $G$ becomes a $\si$-locally compact right topological group. The topological center can be verified to consist of the continuous elements of $\hat{\mathbb{C}_d}$ i.e. $\Ld(G) = \mathbb{T}\times \mathbb{C}\times\mathbb{C}$, where, as before $\mathbb{C}\subset \hat{\mathbb{C}_d}$ via $z\mapsto [t\mapsto e^{2\pi izt}]$. $G$ is hence admissible. 
			
			In this case, we may note that for $H=\mathbb{T}\times{0}\times \{1\}$, $G/H\cong \mathbb{C}\times \mathbb{C}^{AP}$ is a locally compact topological group. Further, no proper subgroup of $H$ gives a topological quotient with respect to $G$, so that by \Cref{sisi}, $H=N(G)$. Again, we have $N(G)$ is compact. A compact strong normal system is given by $G\supset N(G)\supset \{e\}$; the multiplication map  $\mathbb{T}\times \mathbb{C}\times \hat{\mathbb{C}_d} \times \mathbb{T}\times{0}\times \{1\} \to  \mathbb{T}\times \mathbb{C}\times \hat{\mathbb{C}_d}$ is clearly separately continuous as $N(G)\subset \Ld(G)$, whence jointly continuous by Namioka's theorem.
			
			Unlike the last example, $N(G)\subset \Ld(G)$ in this case. Observe that in both examples, $N(G)$ is a compact topological group.
		\end{ex}
		
		\bigskip

		The following example is from \cite{milnes1},\cite{lauloy1}.
		\begin{ex} \label{hehe}
			Let us consider the product topological space $G=\mathbb{T}\times \{\phi,1\}$, where $\phi: \mathbb{T}\to\mathbb{T}$ is a discontinuous automorphism with period 1, i.e. $\phi\circ \phi=1$. The existence of such an automorphism was proven in \cite{milnes2}. Equipped with the multiplication $(u,\ve)(v,\del)=(u\ve(v), \ve\del)$, $G$ is a compact Hausdorff right topological group.	It turns out that $\Ld(G)=N(G)= \mathbb{T}\times \{1\}$, so that $G$ is not admissible. 
			
			 Every element of $C(G)$ may be uniquely decomposed as $f_{\phi}+g_{1}$, where $f, g \in C(\mathbb{T})$.  
			 $M(G)$ may be decomposed in a similar manner. Here we use subscripts $\phi$, 1 to denote elements of $C(\mathbb{T})$, $M(\mathbb{T})$, supported on copies of $\mathbb{T}$ corresponding to $\{1\}$, $\{\phi\}$ respectively. A unique Haar measure on $G$ is given by $\ld_{\phi}/2+\ld_{1}/2$, where $\ld$ is the Haar measure on $\mathbb{T}$.  It was shown in \cite{lauloy1}, that
			\begin{align*}
			\mc&=M(\mathbb{T}\times \{1\})\oplus \mathbb{C} \ld_{\mathbb{T}\times \{\phi\}}\\
			\lgg&=\lc = L_1(\mathbb{T}\times \{1\})\oplus \mathbb{C} 1_{\mathbb{T}\times \{\phi\}}\\
			\mathcal{D}(G)&=C(\mathbb{T}\times \{1\})\oplus \mathbb{C}1_{\mathbb{T}\times \{\phi\}}
			\end{align*}
			
			This follows from the fact that for any $f\in C(\mathbb{T})$, $\mu\in M(\mathbb{T})$, $\delta,\gamma\in \{1,\phi\}$, $(v,\ve)\in G$, 
			\begin{align*}
			f_{\del}\cd \mu_{\gamma}(v,\ve)=\begin{cases}
			\int_{\mathbb{T}} f(t\gamma(v)) d\mu(t) &\text{ if } \ve= \del\circ \gamma\\
			0 &\text{ otherwise}
			\end{cases}
			\end{align*}
			Recall that an element of $\mu\in\mc$ is in $\msc$ if and only $f\cd\mu\in C(g,\si)=C(G/N(G))$. Now if $\mu_{1}+\nu_{\phi}\in \mc$, then, for every $f_1+g_\phi\in C(G)$,
			\begin{align*}
			(f_1+g_\phi)\cd (\mu_{1}+\nu_{\phi})(v,1)
			= f_1\cd\mu_1(v,1)+g_\phi\cd \nu_{\phi}(v,1)
			=\int f(tv) d\mu(t)+\int g(t\phi(v))d\nu(t)
			\end{align*}				
			so that 
			\begin{align*}
			(f_1+g_\phi)\cd (\mu_{1}+\nu_{\phi})(v,1)
			=(f_1+g_\phi)\cd (\mu_{1}+\nu_{\phi})(1,1)
			\end{align*}
			
			for all $f,g\in C(\mathbb{T})$, if and only if $\mu$, $\nu$ are right invariant on $\mathbb{T}$. By the uniqueness of the Haar measure on $\mathbb{T}$, it follows that $\msc=\mathbb{C}\ld_{\mathbb{T}\times \{1\}}\oplus \mathbb{C}\ld_{\mathbb{T}\times \{\phi\}}$. Thus $\mathcal{L}_\si(G) \subsetneq \lc\subset \mw$, and \Cref{heheh} need not hold in general. 	
			Note also that $G$ satisfies the hypothesis of \Cref{subgroup}.
		\end{ex}
		
		\bigskip
		
		Examples of locally compact admissible topological groups are plentiful via taking products of admissible compact right topological groups and locally compact topological groups. In particular, for a locally compact topological group $G$, $G\times G^{D(G)}$ is usually non-trivially right topological, locally compact and admissible. If $G$ is $\si$-compact, this group is also $\si$-locally compact. In this case, note that a strong normal system of subgroups and thus the existence of a Haar measure is guaranteed as $G^{D(G)}$ is a CHART group (\Cref{haarmeas}). However, it still remains an open question as to whether every admissible $\si$-locally compact group possesses a Haar measure or a strong normal system of subgroups.\\

		Many results on compact right topological groups do not generalize easily to locally compact right topological groups. The main one we are interested in is the following:
	 
	 	\begin{prob}
		Do $\si$-locally compact admissible right topological groups have a strong normal system of subgroups? What sufficient conditions ensure a compact system? 
		\end{prob}
	
		More specifically for normal subgroups $L$ of $G$,
		
		\begin{prob}
		If $G$ is $\si$-locally compact admissible is $N(L)\neq L$? When is $N(L)$ compact?
		\end{prob}
	
		The various techniques used to prove this result in the compact setting (\cite{namioka}, \cite{moorsnamioka}, \cite{milnespym2}) do not generalize well to our setting. Although $(L,\si)$ is locally compact by \Cref{baire}, and $N(L)$ is the intersection of all $\si$-closed neighbourhoods of $e$ in $L$, due to the $\si$-topology being non-Hausdorff, the compact sets need not be closed and the compactness of $N(L)$ is not guaranteed.\\
		
		Another problem we have is regarding uniqueness and left-invariance;
		
		\begin{prob}
			If a Haar measure on $G$ exists, is it unique? If $G$ is compact, when is the Haar measure left-$\Ld(G)$-invariant?
			\end{prob}
		
		In the compact case, an admissible right topological group has a faithful representation if and only if it is topological \cite{milnes1}. We thus raise:
		
		\begin{prob}
		Can $\si$-locally compact admissible right topological groups have faithful representations? 
		\end{prob}
		
		In \cite{lauloy2}, it was proven that $B(G)= B(G,\si)\cong B(G/N(G))$, and that certain pleasant properties hold for the same.
		\begin{prob}
		 When do these properties of $B(G)$ hold in the locally compact case?
		\end{prob}

	\Cref{metrizable} makes it uninteresting to consider metrizable locally compact groups unless we are willing to forego admissibility. Examples of non-compact metrizable right topological groups do however exist, as given by Ruppert in \cite{ruppert}. We therefore raise the following
		\begin{pr}
		Do metrizable right topological groups always possess a Haar measure? 
		\end{pr}
	
		We hope to resolve some of these questions in future work.

		\section*{Acknowledgements}
		
		This paper comprises part of the author's PhD thesis. The author would like to thank her PhD supervisor, Dr. Anthony To-Ming Lau for his continued support and guidance.
		
\bibliography{Thesisfinale}
\bibliographystyle{abbrv}
\nocite{*}
	\end{document}